\documentclass[11pt]{article}

\usepackage{macros}
\usepackage[numbers]{natbib}
\usepackage{color}
\definecolor{darkblue}{rgb}{0,0,.4}
\usepackage[colorlinks=true,linkcolor=darkblue,
  citecolor=darkblue,urlcolor=darkblue]{hyperref}
\usepackage[top=2.54cm,left=3.0cm,right=3.0cm,bottom=2.54cm]{geometry}
\usepackage{algorithm}
\usepackage{algorithmicx}
\usepackage{algpseudocode}

\newcommand{\chipone}[2]{\rho\left({#1}|\!|{#2}\right)}
\newcommand{\chipones}[2]{\rho({#1}|\!|{#2})}

\newcommand{\scalarloss}{\ell}

\makeatletter
\newcommand\footnoteref[1]{\protected@xdef\@thefnmark{\ref{#1}}\@footnotemark}
\makeatother
  
\begin{document}

\begin{center}
  {\Large A constrained risk inequality for general losses}\\
  \vspace{.3cm}
  \large{John C.\ Duchi\footnote{\label{footnote:support}
      Supported by NSF CAREER Award
      CCF-1553086 and ONR Young Investigator Award N00014-19-2288}
    ~~~~ Feng Ruan\footnoteref{footnote:support}} \\
  Stanford University \\
  April 2020
\end{center}



\begin{abstract}
  We provide a general constrained risk inequality that applies to arbitrary
  non-decreasing losses, extending a result of Brown and Low
  [\emph{Ann.~Stat.~1996}]. Given two distributions $P_0$ and $P_1$, we find
  a lower bound for the risk of estimating a parameter $\theta(P_1)$ under
  $P_1$ given an upper bound on the risk of estimating the parameter
  $\theta(P_0)$ under $P_0$. The inequality is a useful pedagogical tool, as
  its proof relies only on the Cauchy-Schwartz inequality, it applies to
  general losses, and it transparently gives risk lower bounds on
  super-efficient and adaptive estimators.
\end{abstract}

\section{Introduction}

In the theory of optimality for statistical estimators, we wish to develop
the tightest lower bounds on estimation error possible. With this in mind,
three desiderata make a completely satisfying lower bound: it is
distribution specific, in the sense that the lower bound is a function of
the specific distribution $P$ generating the data; the lower bound is
uniformly achievable, in that there exist estimators achieving the lower
bound uniformly over $P$ in a class $\mc{P}$ of distributions; and there is
a super-efficiency result, so that if an estimator $\what{\theta}$ achieves
better risk than that indicated by the lower bound at a particular
distribution $P_0$, there exist other distributions $P_1$ where the
estimator has worse risk than the bound.  While for problems of estimating
a three or higher-dimensional quantity,
the Stein phenomenon~\cite{Stein56} shows that satisfying all
three of these desiderata is impossible, in the case of estimation of a
real-valued functional $\theta(P)$ of a distribution $P$, one can often
develop such results. It is the purpose of this pedagogical note to show a
transparent proof of such lower bounds via a ``hardest one-dimensional
subproblem'' argument~\cite{Stein56a}. Our hope is that this perspective is
useful for explanation of the failures of super-efficient estimators,
such as the Hodges' estimator, which must achieve inflated error away from
points at which they are superefficient, or for researchers who wish to
simply develop lower bounds in functional estimation.

In classical one-parameter families of distributions, such as location
families or exponential families, the Fisher Information governs estimation
error in a way satisfying our three desiderata of locality, achievability,
and impossibility of super-efficiency, and in classical parametric problems,
no estimator can be super-efficient on more than a set of measure zero
points~\cite{LeCamYa00,VanDerVaart97,VanDerVaart98}. Similarly satisfying
results hold in other problems.  In the case of estimation of the value of a
convex function $f$ in white noise, for example, \citet{CaiLo15} provide
precisely such a result, characterizing a local modulus of continuity with
properties analogous to the Fisher information. For stochastic convex
optimization problems, \citet*{ChatterjeeDuLaZh16} give a computational
analogue of the Fisher Information that governs the difficulty of optimizing
the function.

Key to many of these results, and to understanding nonparametric functional
estimation more broadly, is the \emph{constrained risk inequality} of
\citet{BrownLo96}. \citeauthor{BrownLo96} develop a two-point inequality
that is especially well-suited to providing lower bounds for adaptive
nonparametric function estimation problems, and they also show that it gives
quantitative bounds on the mean-squared error of super-efficient estimators
for one-parameter problems, such as Gaussian mean estimation. Their work,
however, relies strongly on using the squared error loss---that is, the
quality of an estimator $\what{\theta}$ for a parameter $\theta$ is measured
by $\E[(\what{\theta} - \theta)^2]$.  In many applications, it is
interesting to evaluate the error in other metrics, such as absolute error
or the probability of deviation of the estimator $\what{\theta}$ away from
the parameter $\theta$ by more than a specified amount. We extend Brown and
Low's work~\cite{BrownLo96} by providing a constrained risk inequality that
applies to general (non-decreasing) losses. Our proof relies only on the
Cauchy--Schwarz inequality, so we can decouple the argument from the
particular choice of loss. There are more general results on lower bounds
that demonstrate tradeoffs must exist, such as Lepskii's results on
adaptivity in Gaussian white noise models~\cite{Lepskii90} or~\cite[Theorem
  6, App.~A1]{Tsybakov98}.  While (similar to~\cite{BrownLo96}) our approach
does not always provide sharp constants, the constrained risk inequality
allows us to provide finite sample lower bounds for estimation under general
losses, which brings us closer to the celebrated local asymptotic minimax
theorem of Le Cam and
H\'{a}jek~\cite[e.g.][Ch.~8.7]{LeCamYa00,VanDerVaart98}.
To illustrate our results, we provide a applications to estimation of a
normal mean and certain efficient nonparametric estimation problems,
deferring technical proofs to Section~\ref{sec:proofs}.

\section{The constrained risk inequality}

We begin with our setting.  Let $P$ be a
distribution on a sample space $\mc{Z}$, and let $\theta(P) \in \R^k$ be a
parameter of interest. For predicting a point $v \in \R^k$
when the distribution is $P$, the estimator suffers loss
\begin{equation}
  \label{eqn:loss}
  L(v, P) \defeq \scalarloss(\ltwo{v - \theta(P)}),
\end{equation}
where $\scalarloss : \R_+ \to \R_+$ is a non-decreasing scalar
loss function.
For $Z \sim P$ and an estimator $\what{\theta}$ of
$\theta(P)$ based on $Z$,
the risk of $\what{\theta}$ is then
\begin{equation*}
  \risk(\what{\theta}, P)
  \defeq \E_P\left[L(\what{\theta}, P)\right]
  = \E_P\left[\scalarloss(\ltwos{\what{\theta}(Z) - \theta(P)})\right].
\end{equation*}
The result to come relies on the similarity of two distributions to one
another, and accordingly,
we define the $\chi^2$-affinity by
\begin{equation*}
  \chipone{P_1}{P_0}
  \defeq \dchi{P_1}{P_0} + 1
  = \int \frac{dP_1^2}{dP_0}
  = \E_0\left[\frac{dP_1^2}{dP_0^2}\right]
  = \E_1\left[\frac{dP_1}{dP_0}\right],
\end{equation*}
where $\E_0$ and $\E_1$ denote expectation under $P_0$ and $P_1$,
respectively.
With these definitions, we have the following theorem, which gives
a lower bound for the risk of the estimator $\what{\theta}$ on a distribution
$P_1$ given an upper bound for its risk under $P_0$.

\begin{theorem}
  \label{theorem:constrained-risk-convex}
  Assume $\scalarloss : \R_+ \to \R_+$ in the loss~\eqref{eqn:loss}
  is convex. 
  Let $\theta_0 = \theta(P_0)$ and $\theta_1 = \theta(P_1)$, and
  define the separation $\Delta = 2 \scalarloss(\half
  \ltwo{\theta_0 - \theta_1})$.  If the estimator $\what{\theta}$ satisfies
  $\risk(\what{\theta}, P_0) \le \delta$, then
  \begin{equation}
    \label{eqn:constrained-risk-convex}
    \risk(\what{\theta}, P_1)
    \ge 
    \hinge{\Delta^{1/2} - (\chipones{P_1}{P_0} \cdot \delta)^{1/2}}^2.
  \end{equation}
\end{theorem}

A few corollaries are possible. The first applies to more general
(non-convex) loss functions.
\begin{corollary}
  \label{corollary:generic-loss}
  Let the conditions of Theorem~\ref{theorem:constrained-risk-convex} hold,
  except that $\scalarloss : \R_+ \to \R_+$ is an arbitrary non-decreasing
  function. Define $\Delta = \scalarloss(\half \ltwo{\theta_0 -
    \theta_1})$. If the estimator $\what{\theta}$ satisfies
  $\risk(\what{\theta}, P_0) \le \delta$, then
  \begin{equation*}
    \risk(\what{\theta}, P_1)
    \ge \hinge{\Delta^{1/2} - (\chipone{P_1}{P_0} \delta)^{1/2}}^2.
  \end{equation*}
\end{corollary}
\noindent
We can also give a corollary with slightly sharper constants, which applies
to the case that we measure error using a power loss.
\begin{corollary}
  \label{corollary:power-losses}
  In addition to
  the conditions of Theorem~\ref{theorem:constrained-risk-convex},
  assume $\scalarloss(t) = t^k$ for some $k \in (0, \infty)$,
  and define $\Delta = \ltwo{\theta_0 - \theta_1}$.
  If the estimator $\what{\theta}$ satisfies
  $\risk(\what{\theta}, P_0) \le \delta^k$, then
  \begin{equation}
    \label{eqn:constrained-power-loss}
    \risk(\what{\theta}, P_1)
    \ge \begin{cases}
      \hinge{\Delta^{k/2} - (\chipones{P_1}{P_0} \cdot \delta^k)^{1/2}}^2
      & \mbox{if}~ 0 < k \le 2 \\
      \hinge{\Delta - (\chipones{P_1}{P_0} \cdot \delta^2)^{1/2}}^k
      & \mbox{if~} k \ge 2.
    \end{cases}
  \end{equation}
\end{corollary}

\section{Examples}

We provide three examples that apply to estimation of one-dimensional
functionals to illustrate our results. For the first two, we consider
Gaussian mean estimation, where the results are simplest and cleanest to
state, and which immediately demonstrate the failure of the Hodges'
estimator. For the last set of examples, we consider super-efficient
estimation in a general family of nonparametric models.

\subsection{Gaussian mean estimation}

We provide two examples that apply to one-dimensional Gaussian mean
estimation to illustrate our results.  For the
first, we consider a zero-one loss function
indicating whether the estimated mean is near the
true mean.
Fix $\sigma^2 > 0$ and let $X_1,
\ldots, X_n$ be i.i.d.\ $P_\theta = \normal(\theta, \sigma^2)$, and let
$\scalarloss(t) = \indic{|t| \ge \sigma / \sqrt{n}}$, so that
\begin{equation*}
  \risk(\what{\theta}, P_\theta^n)
  = P_\theta^n\left(|\what{\theta}(X_1, \ldots, X_n)
  - \theta| \ge \frac{\sigma}{\sqrt{n}}\right),
\end{equation*}
where $P_\theta^n$ denotes the $n$-fold product of $X_i \simiid
\normal(\theta, \sigma^2)$.  Now, let $\delta_n \in [0, 1], \delta_n \to 0$
be an otherwise arbitrary sequence, and let $0 < c < 1$ be a fixed constant.
Define the sequence of local
alternative parameter spaces
\begin{equation*}
  \Theta_n \defeq \left\{\theta \in \R
  \mid 2 \frac{\sigma}{\sqrt{n}}
  \le |\theta| \le
  \frac{\sigma}{\sqrt{n}} \sqrt{c \log \frac{1}{\delta_n}}
  \right\}.
\end{equation*}
We then have the following proposition.
\begin{proposition}
  \label{proposition:threshold-loss}
  Let $\what{\theta}_n : \R^n \to \R$ be a sequence of estimators
  satisfying $\risk(\what{\theta}_n, P_0^n) \le \delta_n$
  for all $n$. Then
  \begin{equation*}
    \liminf_n \inf_{\theta \in \Theta_n}
    \risk(\what{\theta}_n, P_\theta^n)
    = \liminf_n \inf_{\theta \in \Theta_n}
    P_\theta^n\left(\sqrt{n} |\what{\theta}_n(X_1, \ldots, X_n) - \theta|
    \ge \sigma \right) = 1.
  \end{equation*}
\end{proposition}
\begin{remark}
  The Le Cam--H\'{a}jek asymptotic
  minimax theorem (cf.~\cite{VanDerVaart97,LeCamYa00}) implies that for any
  symmetric, quasiconvex loss $\scalarloss : \R^k \to \R_+$, if
  $\{P_\theta\}_{\theta \in \Theta}$ is a suitably regular family of
  distributions with Fisher information matrices $I_\theta$, then for any
  $\theta_0 \in \interior \Theta$ there exist sequences of prior densities
  $\pi_{n,c}$ supported on $\{\theta \in \R^k \mid \ltwo{\theta - \theta_0}
  \le c / \sqrt{n}\}$ such that
  \begin{equation}
    \liminf_{c \to \infty}
    \liminf_n \inf_{\what{\theta}_n}
    \int \E_\theta[\scalarloss(\sqrt{n}(\what{\theta}_n - \theta))]
    d\pi_{n,c}(\theta)
    \ge \E[\scalarloss(Z)]
    ~~ \mbox{where} ~~
    Z \sim \normal(0, I_{\theta_0}^{-1})
    \label{eqn:local-minimax}
  \end{equation}
  (see \cite[Lemma 6.6.5]{LeCamYa00}
  and also~\cite[Eq.~(9)]{VanDerVaart97}).
  This in turn implies that for Lebesgue-almost-all $\theta$, we have
  $\limsup_n \E_\theta[\scalarloss(\sqrt{n}(\what{\theta}_n - \theta))]
  \ge \E[\scalarloss(Z)]$ for $Z \sim \normal(0, I_\theta^{-1})$.
  For the indicator loss $\loss(t) = \indic{|t| \ge \sigma}$, these
  results imply that $\limsup_n P_\theta(|\sqrt{n}(\what{\theta}_n -
  \theta)| \ge \sigma) \ge 2 \Phi(-1)$ for almost all $\theta$ in our normal
  mean setting, where $\Phi$ is the standard normal
  CDF. Proposition~\ref{proposition:threshold-loss} strengthens this: if
  there exists a point of super-efficiency with asymptotic probability of
  error 0, then there exists a large set of points with asymptotic
  probability of error 1.
\end{remark}

\begin{proof}
  Assume that $n$ is large enough
  that $c \log \frac{1}{\delta_n} \ge 2$,
  and let $\theta \in \Theta_n$. A
  calculation then yields that
  \begin{equation*}
    \chipones{P_\theta^n}{P_0^n}
    = \exp\left(\frac{n \theta^2}{\sigma^2}\right)
    \le
    \exp\left(\frac{c \sigma^2 n \log \frac{1}{\delta_n}}{\sigma^2 n}\right)
    = \delta_n^{-c}.
  \end{equation*}
  We also have that $\scalarloss(\half|\theta|)
  = \indic{|\theta| \ge 2 \sigma / \sqrt{n}} =
  1$, and substituting this into Corollary~\ref{corollary:generic-loss}, we
  obtain $\risk(\what{\theta}, P_\theta^n) \ge \hinge{1 - \delta_n^{1-c}}^2$.
  As $c < 1$, this quantity tends to 1 as $n \to \infty$.
\end{proof}

Let us consider Corollary~\ref{corollary:power-losses} for our second
application. In this case, we consider estimating a Gaussian mean given $X_i
\simiid \normal(\theta, 1)$, but we use the absolute error $L(\theta, P) =
|\theta - \theta(P)|$ as our loss as opposed to the typical mean squared
error.
\begin{proposition}
  \label{proposition:absolute-loss}
  Let $\what{\theta} : \R^n \to \R$ be an estimator such that
  $\risk(\what{\theta}, P_0^n) \le \frac{\epsilon}{\sqrt{n}}$.
  Then for all $\alpha \in [0, 1]$, there exists
  $\theta$ such that
  \begin{equation*}
    \risk(\what{\theta}, P_\theta^n)
    \ge \sqrt{\frac{\alpha}{n}}
    \hinge{\sqrt[4]{\log \frac{1}{\epsilon}}
      - \sqrt[4]{\frac{\epsilon^{2 - 2\alpha}}{\alpha}}}^2.
  \end{equation*}
  In particular, if $\epsilon \le 10^{-2}$, then there exists
  $\theta$ with
  $\risk(\what{\theta}, P_\theta^n)
  \ge \frac{1}{4} \sqrt{\frac{\log \frac{1}{\epsilon}}{n}}$.
\end{proposition}
\begin{proof}
  Let $\alpha \in [0, 1]$, to be chosen presently. Let $\theta \ge 0$ with
  $\theta^2 = \frac{\alpha \log\frac{1}{\epsilon}}{n}$.  Then we have
  $\chipones{P_\theta^n}{P_0^n} = \exp(n \theta^2) =
  \frac{1}{\epsilon^\alpha}$ and that $\Delta = |\theta|$ in the notation of
  Corollary~\ref{corollary:power-losses}. The corollary then implies
  \begin{equation*}
    \risk(\what{\theta}, P_\theta^n)
    \ge \hinge{
      \sqrt{\theta} -
      \sqrt{\epsilon^{-\alpha} \epsilon / \sqrt{n}}}^2
    = \frac{\sqrt{\alpha}}{\sqrt{n}}
    \hinge{\sqrt[4]{\log\frac{1}{\epsilon}} -
      \sqrt[4]{\frac{\epsilon^{2 - 2 \alpha}}{\alpha}}}^2.
  \end{equation*}
  The second result of the proposition follows by taking
  $\alpha = 1/8$ and using the numerical fact that
  that $\sqrt[4]{\log\frac{1}{\epsilon}}
  - \sqrt[4]{8 \epsilon^{7/4}} \ge \sqrt[4]{\log\frac{1}{\epsilon}/2}$
  for $\epsilon \le 10^{-2}$.
\end{proof}

\newcommand{\hodges}{\what{\theta}_n^{\textup{Hodges}}}

As an example consequence of Proposition~\ref{proposition:absolute-loss},
consider the Hodges' estimator
\begin{equation*}
  \hodges \defeq \begin{cases} \wb{X}_n & \mbox{if}~ |\wb{X}_n| \ge n^{-1/4} \\
    0 & \mbox{otherwise},
  \end{cases}
\end{equation*}
where $\wb{X}_n \defeq \frac{1}{n}
\sum_{i = 1}^n X_i$.
At $\theta = 0$, this estimator satisfies
\begin{equation*}
  \E[|\hodges|]
  = \E[|\wb{X}_n| \indics{|\wb{X}_n| \ge n^{-1/4}}]
  \le \sqrt{\frac{1}{n}} \cdot
  \sqrt{P_0(|\wb{X}_n| \ge n^{-1/4})}
  \le \sqrt{\frac{2}{n} \exp\left(-\frac{\sqrt{n}}{2}\right)}
\end{equation*}
by the standard tail bound that $\P(|Z| \ge t) \le 2 \exp(-t^2 / 2 \sigma^2)$
for $Z \sim \normal(0, \sigma^2)$. In particular,
for all large enough $n$, there is a $\theta \in [0, n^{-1/2}]$ such that
\begin{equation*}
  \E_\theta[|\hodges - \theta|] \ge \frac{1}{8 n^{1/4}} \gg \frac{1}{\sqrt{n}}.
\end{equation*}

\subsection{Super-efficient estimation in nonparametric models}
\label{sec:nonparametric-efficiencies}

\newcommand{\effinluence}{\dot{\theta}_0}

It is often interesting to derive efficiency lower bounds outside of
standard parametric models; it is our experience that students are
frequently curious about such quantities, especially when they have seen
only Fisher-information-based lower bounds. Conveniently, we can also apply
our results to estimation of functionals in general non-parametric
models. In this case, we focus on quantities where the classical
asymptotic normality results apply, so that there do indeed exist
classically efficient estimators and an analogue of the Le Cam--H\'{a}jek
local asymptotic minimax theorems. We first present a general result that
applies to appropriately smooth parameters of the underlying distribution,
which we subsequently specialize to estimation of the mean of an arbitrary
distribution with finite variance. We adapt the classical idea of
Stein~\cite{Stein56a}, which constructs hardest one-dimensional subproblems,
following the treatment of~\citet[Chapter 25]{VanDerVaart98}.

To set the stage, consider estimation of a parameter $\theta(P_0) \in \R$ of
a distribution $P_0$ on the space $\mc{Z}$.  Letting $\mc{P}$ denote the
collection of all distributions on $\mc{Z}$, we consider sub-models
$\mc{P}_0 \subset \mc{P}$ around $P_0$ defined in terms of local
perturbations of $P_0$. In particular, let $\mc{G} \subset L^2(P_0)$ consist
of those functions $g : \mc{Z} \to \R$ satisfy $\E_0[g(Z)] = 0$ and
$\E_0[g(Z)^2] < \infty$.  For bounded functions $g \in \mc{G}$, we may
consider tilts of the distribution $P_0$ of the form
\begin{equation*}
  dP(z) = (1 + t g(z)) dP_0(z)
\end{equation*}
for small $t$; however, as $g$ may be unbounded, we require a bit more care.
Following~\cite[Example 25.16]{VanDerVaart98}, we let $\phi : \R \to [0, 2]$
be any $\mc{C}^3$ function satisfying $\phi(1) = 1$, $\phi'(1) = 1$, and for
which both $\linf{\phi'} \le K$ and $\linf{\phi''} \le K$ for a constant
$K$; for example, $\phi(t) = 2 / (1 + e^{-2t})$ suffices. For any $g \in
\mc{G}$, define the tilted distribution
\begin{equation}
  \label{eqn:tilted-distribution}
  dP_{t,g}(z) \defeq \frac{1}{C_t} \phi(t g(z)) dP_0(z)
  ~~ \mbox{where} ~~
  C_t = \int \phi(t g(z)) dP_0(z).
\end{equation}
The following lemma describes the divergence of $P_{t,g}$
from $P_0$ (see
Section~\ref{sec:proof-tilt-dists} for proof).
\begin{lemma}
  \label{lemma:tilt-dists}
  Let $g \in \mc{G}$ and $P_0$ and $P_{t,g}$ be as defined in
  Eq.~\eqref{eqn:tilted-distribution}. Then
  \begin{equation*}
    \dchi{P_{t,g}}{P_0}
    = 1 + t^2 \E_0[g(Z)^2] + o(t^2)
    ~~ \mbox{and} ~~
    |C_t - 1| \le
    \frac{K}{2} t^2 \E_0[g(Z)^2].
  \end{equation*}
\end{lemma}

With this setting, let us assume that our parameter $\theta$ of interest is
smooth in the underlying perturbation~\eqref{eqn:tilted-distribution},
meaning that there exists an \emph{influence function} $\effinluence :
\mc{Z} \to \R$, $\effinluence \in L^2(P_0)$, with $\E_0[\effinluence(Z)] =
0$ such that
\begin{equation}
  \label{eqn:efficient-influence}
  \theta(P_{t,g}) = \theta(P_0) + t \E_0[\effinluence(Z) g(Z)]
  + o(t)
\end{equation}
as $t \to 0$, that is, $\theta(P_{t,g})$ has a linear first-order expansion
in $L^2$ based on $\effinluence$. For example,
the mean $\theta(P) = \E_P[Z]$ has
the identity mapping $\effinluence(Z) = Z - \E_P[Z]$.
For more on such linear expansions and
their importance and existence, see~\cite[Chapter~25]{VanDerVaart98}. In
short, however, the influence function allows extension of the Fisher
Information from classical problems, and by defining $I_0^{-1} \defeq
\E_{P_0}[\effinluence(Z)^2]$, one has the analogue of the local minimax
lower bound~\eqref{eqn:local-minimax} that there exist sequences of prior
densities $\pi_n$ supported on $\{t \in \R \mid |t| \le 1 / \sqrt{n}\}$ such
that
\begin{equation}
  \label{eqn:nonparametric-local-minimax}
  \sup_{g \in \mc{G}}
  \liminf_n \inf_{\what{\theta}_n}
  \int \E_{P_{t,g}^n}[\loss(\sqrt{n}(\what{\theta}_n - \theta(P_{t,g})))]
  d\pi_n(t) \ge \E[\loss(Z)]
  ~~ \mbox{where} ~~
  Z \sim \normal(0, I_0^{-1}).
\end{equation}
The supremum above may be taken to be over only scalar multiples
of the function $\effinluence$.

\subsubsection{Non-convergence in probability: the general case}

We now come to our super-efficiency result, which we will specialize to the
nonparametric mean presently. Essentially the weakest typical form of
convergence of estimators is convergence in probability, which is of course
implied by convergence in mean-square or absolute error. As our general
constrained risk inequality (Corollary~\ref{corollary:generic-loss}) handles
this case without challenge, and because lower bounds on the probability of
error are strong, we focus on the zero-one error.
Let $K < \infty$ be an
arbitrary constant, and for each $n$, define the loss function $\loss(t) =
\indics{\sqrt{n} |t| \ge K}$, so that
\begin{equation*}
  \risk(\what{\theta}, P^n) = P^n\left(\sqrt{n} |\what{\theta}(Z_1, \ldots,
  Z_n) - \theta(P)| \ge K\right).
\end{equation*}
Under the assumption that $\what{\theta}_n$ is a super-efficient sequence of
estimators under $P_0$, we will show that for essentially \emph{all}
non-trivial local alternatives, defined by the
tilting~\eqref{eqn:tilted-distribution}, the estimators $\what{\theta}_n$
have probability of error tending to 1.

Making this more precise, consider the subset
\begin{equation}
  \label{eqn:non-trivial-perturbations}
  \mc{G}_0 \defeq \{g \in \mc{G} \mid \E_0[\effinluence(Z) g(Z)] \neq 0,
  \E_0[g(Z)^2] \le 1 \},
\end{equation}
that is, those functions $g \in \mc{G}$ for which the perturbation of
$\theta(P_0)$ to $\theta(P_{t,g})$ is non-trivial as $t \to 0$, by the
first-order expansion~\eqref{eqn:efficient-influence}.  Let us suppose that
$\risk(\what{\theta}_n, P_0^n) \le \delta_n$ for all $n$, where $\delta_n
\to 0$ and $\frac{1}{n} \log \frac{1}{\delta_n} \to 0$ (this last assumption
is simply to make our argument simpler).
Now, let $B > 2$ and $c \in (0, 1)$ be otherwise arbitrary
constants, and for each $g \in \mc{G}_0$,
define the set of local alternative distributions
\begin{equation}
  \mc{P}_{n,g} \defeq
  \left\{P_{t,g} \in \mc{P}
  \mid
  \frac{K^2}{n}
  \frac{B^2}{\E_0[\effinluence(Z) g(Z)]^2}
  \le t^2 \le \frac{c}{n}
  \log \frac{1}{\delta_n} \right\}.
  \label{eqn:tilted-family}
\end{equation}
We have the following proposition.
\begin{proposition}
  \label{proposition:nonparametric-general}
  Let $\what{\theta}_n : \mc{Z}^n \to \R$ be a sequence of estimators
  satisfying $\risk(\what{\theta}_n, P_0^n) \le \varepsilon_n$, where
  $\epsilon_n \to 0$.
  Let $\delta_n \ge \varepsilon_n$ be any sequence
  satisfying
  $\delta_n \to 0$ and $n^{-1} \log \delta_n \to 0$.
  Then
  \begin{equation*}
    \inf_{g \in \mc{G}_0} \liminf_n \inf_{P \in \mc{P}_{n,g}}
    \risk(\what{\theta}_n, P^n)
    = \inf_{g \in \mc{G}_0}
    \liminf_n \inf_{P \in \mc{P}_{n,g}}
    P^n \left(\sqrt{n} |\what{\theta}_n(Z_1, \ldots, Z_n)
    - \theta(P)| \ge K \right) = 1.
  \end{equation*}
\end{proposition}
\begin{remark}
  This result parallels Proposition~\ref{proposition:threshold-loss},
  applying to nonparametric estimators. In comparison with the local
  asymptotic minimax result~\eqref{eqn:nonparametric-local-minimax}, we see
  the stronger result that super-efficiency at a single distribution
  for the zero-one error implies that asymptotically, the loss is as large
  as possible for a wide range of alternative distributions.
\end{remark}
\begin{proof}
  Fix $g \in \mc{G}_0$, and let $\theta_t = \theta(P_{t,g})$
  and $\theta_0 = \theta(P_0)$ be parameters of interest. For
  shorthand, define $\Delta = \E_0[\effinluence(Z) g(Z)] \neq 0$, so that
  $\theta_t = \theta_0 + (1 + o(1)) t \Delta$ as $t \to 0$. By
  Lemma~\ref{lemma:tilt-dists}, we have that
  \begin{equation*}
    \chipones{P_{t,g}^n}{P_0^n}
    = \left(1 + (1 + o(1)) t^2 \E_0[g(Z)^2]\right)^n
  \end{equation*}
  as $t \to 0$, so that if $\E_0[g(Z)^2] \le 1$,
  \begin{align}
    \sup_t
    \left\{
    \chipones{P_{t,g}^n}{P_0^n}
    \mid t^2 \le \frac{c}{n} \log \frac{1}{\delta_n}\right\}
    & \le \left(1 + (1 + o(1)) \frac{c}{n} \log \frac{1}{\delta_n}\right)^n
    \nonumber \\
    & \le \exp\left((1 + o(1)) c \log \frac{1}{\delta_n}\right)
    = \delta_n^{-c + o(1)}
    \label{eqn:delta-sup-power}
  \end{align}
  as $n \to \infty$.
  Note that as $B > 2$, by the 
  definition~\eqref{eqn:efficient-influence} of an influence function,
  we have for all $t$ satisfying $\frac{B K}{|\Delta|}
  \le \sqrt{n} |t| \le \sqrt{c \log \frac{1}{\delta_n}}$ that
  \begin{align*}
    \loss(|\theta_t - \theta_0| / 2)
    = \indic{\sqrt{n}|\theta_t - \theta_0| \ge 2 K}
    & = \indic{\sqrt{n} \left|\frac{B K}{\sqrt{n}}
      (1 + o(1)) \Delta\right| \ge 2 K} \\
    & = \indic{|B K \pm o(1)| \ge 2 K}
    = 1 ~~ \mbox{for~large~enough~} n,
  \end{align*}
  where the final equality holds because $B > 2$.
  Applying Corollary~\ref{corollary:generic-loss} and
  inequality~\eqref{eqn:delta-sup-power}, we thus obtain
  for large enough $n$, all $P \in \mc{P}_{n,g}$ satisfy
  \begin{equation*}
    \risk(\what{\theta}_n, P^n)
    \ge \hinge{1 - \sqrt{\delta_n^{-c + o(1)} \delta_n}}^2,
  \end{equation*}
  which tends to $1$ as $n \to \infty$ because $\delta_n \to 0$ and
  $c < 1$.
\end{proof}

\subsubsection{Non-convergence in probability for the mean}

Proposition~\ref{proposition:nonparametric-general} is abstract, so we make
it more concrete by considering mean estimation for distributions with
variance 1.  Let $P_0$ be a distribution on $\R$ with $\E_0[Z] = 0$ and
$\var_0(Z) = 1$. In this case, the influence function is the identity
mapping $\effinluence(z) = z$. Let $0 < K < \infty$ be any constant. In this
case, the family $\mc{G}_0$ of non-trivial
perturbations~\eqref{eqn:non-trivial-perturbations} is precisely those with
non-zero covariance with the random variable $Z$,
\begin{equation*}
  \mc{G}_0 =
  \left\{ g : \R \to \R \mid
  \E_0[g(Z)] = 0, \E_0[g(Z)^2] \le 1,
  ~ \mbox{and} ~
  \E_0[Z g(Z)] \neq 0
  \right\}.
\end{equation*}
We thus have the following corollary, which applies to the tilted families
$\mc{P}_{n,g}$ as above~\eqref{eqn:tilted-family}.
\begin{corollary}
  Let $\what{\theta}_n : \mc{Z}^n \to \R$ be any sequence of estimators
  such that $P_0^n(\sqrt{n}|\what{\theta}_n| \ge K) \le \varepsilon_n$,
  where $\epsilon_n \to 0$.
  Let $\delta_n \ge \varepsilon_n$ be any sequence
  satisfying
  $\delta_n \to 0$ and $n^{-1} \log \delta_n \to 0$.
  Then
  \begin{equation*}
    \inf_{g \in \mc{G}_0} \liminf_n \inf_{P \in \mc{P}_{n,g}}
    P^n\left(\sqrt{n} |\what{\theta}_n - \E_P[Z]| \ge K\right) = 1.
  \end{equation*}
\end{corollary}
\noindent
In short, we see the expected result: if any estimator achieves even the
in-probability convergence $\what{\theta}_n = o_P(1/\sqrt{n})$ at $\theta =
0$, then there must be a large collection of distributions where the
\emph{best} performance of the estimator across the entire collection must
be worse than the typical $\sqrt{n}$-rate of convergence.


\section{Discussion}

We have provided an extension of Brown and Low's constrained risk
inequality~\cite{BrownLo96}, showing how to provide risk inequalities for
general losses.  Our results on efficient non-parametric estimators in
Section~\ref{sec:nonparametric-efficiencies} immediately extend beyond 0-1
losses. For example, consider estimating a parameter $\theta(P_0)$ of a
distribution $P_0$ where $\theta$ has influence function $\effinluence : \R
\to \R$, and assume the estimator sequence $\what{\theta}_n : \R^n \to \R$
satisfies
\begin{equation*}
  \E_{P_0^n}\left[|\what{\theta}_n - \theta(P_0)|\right]
  \le \sqrt{\frac{\delta_n}{n}}
\end{equation*}
where $\delta_n \to 0$. Then for the family $\mc{G}_0$ consisting
of $g : \R \to \R$ with $\E_0[g(Z)] = 0$, $\E_0[g(Z)^2] \le 1$, and
$\E_0[\effinluence(Z) g(Z)] \neq 0$, we can consider an analogue of the tilted
family~\eqref{eqn:tilted-family} where for $0 < c_0 < c_1 < 1$ we define
\begin{equation*}
  \mc{P}_{n,g} = \left\{P_{t,g} \mid
  c_0 \frac{\log \frac{1}{\delta_n}}{n} \le t^2 \le
  c_1 \frac{\log \frac{1}{\delta_n}}{n} \right\}.
\end{equation*}
Then by Corollary~\ref{corollary:power-losses} and an argument
analogous to that for Proposition~\ref{proposition:absolute-loss}, there
exists a numerical constant $K > 0$ such that for
all $g \in \mc{G}_0$,
\begin{equation*}
  \liminf_n \inf_{P \in \mc{P}_{n,g}}
  \sqrt{\frac{n}{\log \frac{1}{\delta_n}}}
  \E_{P^n}\left[|\what{\theta}_n - \theta(P)|\right]
  \ge K |\E_0[\effinluence(Z) g(Z)]| > 0.
\end{equation*}

The one-dimensional lower bounds we have provided are, we hope,
transparent---relying only on the Cauchy-Schwarz inequality---and easy to
apply to a range of estimation settings, making them well-suited to
pedagogical situations. It is possible to follow \citeauthor{BrownLo96}'s
work~\cite{BrownLo96} to give non-adaptivity results in nonparametric
function estimation~\cite[cf.]{DonohoJo95, Lepskii90, Lepskii91, BrownLo96,
  Tsybakov98}, with relatively straightforward derivations (though of
course, these results are known).
We hope that our constrained risk inequalities for general
losses may lead to easier understanding of such issues in
other areas as well.


\section{Proofs}
\label{sec:proofs}

\subsection{Proof of Theorem~\ref{theorem:constrained-risk-convex}}
\label{sec:proof-constrained-risk-convex}

It is no loss of generality to assume that $\what{\theta}(z) \in [\theta_0,
  \theta_1] = \{t \theta_0 + (1 - t) \theta_1 \mid t \in [0, 1]\}$ for all
$z$: letting $\mbox{proj}(\theta) = \argmin_{\theta'} \{\ltwo{\theta -
  \theta'} \mid \theta' \in [\theta_0, \theta_1]\}$ be the projection of
$\theta$ onto the segment $[\theta_0, \theta_1]$, then
$\ltwo{\mbox{proj}(\theta) - \theta_i} \le \ltwo{\theta - \theta_i}$ for $i
\in \{0, 1\}$ by standard properties of convex
projections~\cite{HiriartUrrutyLe93ab}.

For any $\theta \in [\theta_0, \theta_1]$, which must
satisfy $\theta = t \theta_0 + (1 - t) \theta_1$,
we have
\begin{align}
  \lefteqn{\sqrt{\scalarloss(\ltwo{\theta - \theta_0})}
    + \sqrt{\scalarloss(\ltwo{\theta - \theta_1})}
    = \sqrt{\scalarloss((1 - t) \ltwo{\theta_0 - \theta_1})}
    + \sqrt{\scalarloss(t \ltwo{\theta_0 - \theta_1})}} \nonumber \\
  & \qquad\qquad\qquad\qquad ~
  \ge \sqrt{\scalarloss((1 - t) \ltwo{\theta_0 - \theta_1})
    + \scalarloss(t \ltwo{\theta_0 - \theta_1})}
  \ge \sqrt{2 \scalarloss\left(\half\ltwo{\theta_0 - \theta_1}\right)}
  \label{eqn:majorization-convex}
\end{align}
as $\scalarloss(t a) + \scalarloss((1 - t) a)$ is minimized by $t = \half$
for any $a \ge 0$.  Using the majorization
inequality~\eqref{eqn:majorization-convex} and our without loss of
generality assumption that $\what{\theta}(z) \in [\theta_0, \theta_1]$ for
all $z \in \mc{Z}$, we thus have
\begin{equation}
  \label{eqn:majorization-summed-risk-lb}
  \E_1\left[\scalarloss(\ltwos{\what{\theta} - \theta_0})^{1/2}\right]
  + \E_1\left[\scalarloss(\ltwos{\what{\theta} - \theta_1})^{1/2}\right]
  \ge \sqrt{2 \scalarloss\left(\half\ltwo{\theta_0 - \theta_1}\right)}
  = \Delta^{1/2}.
\end{equation}
Now, using the Cauchy--Schwarz inequality and rearranging
inequality~\eqref{eqn:majorization-summed-risk-lb}, we have
\begin{equation*}
  \risk(\what{\theta}, P_1)
  \ge \E_1\left[\scalarloss(\ltwos{\what{\theta} - \theta_1})\right]
  \ge \E_1\left[\scalarloss(\ltwos{\what{\theta} - \theta_1})^{1/2}\right]^2
  \ge \hinge{\Delta^{1/2} -
    \E_1\left[\scalarloss(\ltwos{\what{\theta} - \theta_0})^{1/2}\right]}^2.
\end{equation*}
Finally, a likelihood ratio change of measure yields
that
\begin{align*}
  \E_1\left[\scalarloss(\ltwos{\what{\theta} - \theta_0})^{1/2}\right]
  & = \E_0\left[\frac{dP_1}{dP_0}
    \scalarloss(\ltwos{\what{\theta} - \theta_0})^{1/2}\right] \\
  & \le \E_0\left[\frac{dP_1^2}{dP_0^2}\right]^{1/2}
  \E_0\left[\scalarloss(\ltwos{\what{\theta} - \theta_0})\right]^{1/2}
  = \left(\chipone{P_1}{P_0} \risk(\what{\theta}, P_0)\right)^{1/2}.
\end{align*}
This gives the lower bound~\eqref{eqn:constrained-risk-convex}
once we use that $\risk(\what{\theta}, P_0) \le \delta$.

\subsection{Proof of Corollary~\ref{corollary:generic-loss}}
\label{sec:proof-corollary-generic-loss}

The proof is nearly identical to that of
Theorem~\ref{theorem:constrained-risk-convex}, with one minor
change. Instead of the majorization
inequality~\eqref{eqn:majorization-convex},
we have for all $t \in [0, 1]$ that
\begin{equation*}
  \scalarloss(t \ltwo{\theta_0 - \theta_1})
  + \scalarloss((1 - t) \ltwo{\theta_0 - \theta_1})
  \ge \scalarloss\left(\half \ltwo{\theta_0 - \theta_1}\right).
\end{equation*}
Substituting this and the definition $\Delta = \scalarloss(\half
\ltwo{\theta_0 - \theta_1})$, then following the proof of
Theorem~\ref{theorem:constrained-risk-convex}, \emph{mutatis mutandis},
gives the corollary.

\subsection{Proof of Corollary~\ref{corollary:power-losses}}

The proof is again identical to
Theorem~\ref{theorem:constrained-risk-convex}, except that we consider
separately the cases $k \in \openleft{0}{2}$ and $k > 2$.  In the first case
that $0 < k \le 2$, we replace the majorization
inequality~\eqref{eqn:majorization-convex} for $\theta = t \theta_0 + (1 -
t) \theta_1$, where $t \in [0,1]$, with the inequality
\begin{equation*}
  L(\theta, P_0)^{1/2}
  + L(\theta, P_1)^{1/2}
  = \left[(1 - t)^{k/2} + t^{k/2}\right]\ltwo{\theta_0 - \theta_1}^{k/2}
  \ge \ltwo{\theta_0 - \theta_1}^{k/2}.
\end{equation*}
Using $\Delta = \ltwo{\theta_0 - \theta_1}$ and tracing the proof
of Theorem~\ref{theorem:constrained-risk-convex} then gives 
the first inequality~\eqref{eqn:constrained-power-loss}.
For the second inequality, the case $k \in (2, \infty)$,
we may apply the first case that $k \le 2$ and H\"older's
inequality.
Indeed, by the
assumption that $\risk(\what{\theta}, P_0) \le \delta^k$,
we have
\begin{equation*}
  \E_0\left[\ltwos{\what{\theta} - \theta_0}^2\right]
  \le \E_0\left[\ltwos{\what{\theta} - \theta_0}^k\right]^{2/k}
  \le \delta^2.
\end{equation*}
Applying the result for $k = 2$ in the first case of
inequality~\eqref{eqn:constrained-power-loss} yields
\begin{equation*}
  \risk(\what{\theta}, P_1)
  \ge \E_1\left[\ltwos{\what{\theta} - \theta_1}^2\right]^{k/2}
  \geq \hinge{\Delta - (\chipones{P_1}{P_0}\delta^2)^{1/2}}^k.
\end{equation*}

\subsection{Proof of Lemma~\ref{lemma:tilt-dists}}
\label{sec:proof-tilt-dists}

By the boundedness assumptions
on $\phi'$ and $\phi''$, Taylor's theorem implies that
\begin{equation*}
  |\phi(t) - 1|
  \le \linf{\phi'} |t|
  \le K |t|
  ~~ \mbox{and} ~~
  |\phi(t) - 1 - t|
  \le \half \linf{\phi''} t^2
  \le \half K t^2
\end{equation*}
for all $t \in \R$.
Thus we have
\begin{equation*}
  C_t = \int \phi(t g(z)) dP_0(z)
  = \int (1 + t g(z)) dP_0(z)
  \pm \frac{K}{2} \int t^2 g(z)^2 dP_0(z)
  = 1 \pm \frac{Kt^2}{2} \E_0[g(Z)^2].
\end{equation*}
Let $\sigma^2 = \E_0[g(Z)^2]$ for shorthand.
Considering the $\chi^2$-divergence, we have
$\dchi{P_{t,g}}{P_0} =
\int (\phi(t g(z)) / C_t - 1)^2 dP_0(z)$, and the integrand
has the bound
\begin{equation*}
  \left(\frac{\phi(t g(z))}{C_t} - 1\right)^2
  \le 
  \left(\frac{1 + K |t g(z)|}{1 - K t^2 \sigma^2} - 1\right)^2
  \le \frac{2 K^2 t^2}{(1 - Kt^2 \sigma^2)^2}
  (g(z)^2 + t^2 \sigma^4),
\end{equation*}
and
\begin{equation*}
  \lim_{t \to 0}
  \frac{1}{t^2}
  \left(\frac{\phi(t g(z))}{C_t} - 1\right)^2
  = \lim_{t \to 0}
  \frac{1}{t^2}
  \left(\frac{1 + t g(z) + O(t^2)}{1 - O(t^2)} - 1\right)^2
  = g(z)^2.
\end{equation*}
Lebesgue's dominated convergence theorem implies that
$\lim_{t \to 0} \frac{1}{t^2} \dchi{P_{t,g}}{P_0} = \E_0[g(Z)^2]$, as desired.

\setlength{\bibsep}{1pt}
\bibliography{bib}

\begin{thebibliography}{14}
\providecommand{\natexlab}[1]{#1}
\providecommand{\url}[1]{\texttt{#1}}
\expandafter\ifx\csname urlstyle\endcsname\relax
  \providecommand{\doi}[1]{doi: #1}\else
  \providecommand{\doi}{doi: \begingroup \urlstyle{rm}\Url}\fi

\bibitem[Brown and Low(1996)]{BrownLo96}
L.~D. Brown and M.~G. Low.
\newblock A constrained risk inequality with applications to nonparametric
  functional estimation.
\newblock \emph{Annals of Statistics}, 24\penalty0 (6):\penalty0 2524--2535,
  1996.

\bibitem[Cai and Low(2015)]{CaiLo15}
T.~Cai and M.~Low.
\newblock A framework for estimating convex functions.
\newblock \emph{Statistica Sinica}, 25:\penalty0 423--456, 2015.

\bibitem[Chatterjee et~al.(2016)Chatterjee, Duchi, Lafferty, and
  Zhu]{ChatterjeeDuLaZh16}
S.~Chatterjee, J.~Duchi, J.~Lafferty, and Y.~Zhu.
\newblock Local minimax complexity of stochastic convex optimization.
\newblock In \emph{Advances in Neural Information Processing Systems 29}, 2016.

\bibitem[Donoho and Johnstone(1995)]{DonohoJo95}
D.~L. Donoho and I.~M. Johnstone.
\newblock Adapting to unknown smoothness via wavelet shrinkage.
\newblock \emph{Journal of the American Statistical Association}, 90\penalty0
  (432):\penalty0 1200--1224, 1995.

\bibitem[Efromovich(1999)]{Efromovich99}
S.~Efromovich.
\newblock \emph{Nonparametric Curve Estimation: Methods, Theory, and
  Applications}.
\newblock Springer-Verlag, 1999.

\bibitem[Hiriart-Urruty and Lemar\'echal(1993)]{HiriartUrrutyLe93ab}
J.~Hiriart-Urruty and C.~Lemar\'echal.
\newblock \emph{Convex {A}nalysis and {M}inimization {A}lgorithms {I} \& {II}}.
\newblock Springer, New York, 1993.

\bibitem[{Le Cam} and Yang(2000)]{LeCamYa00}
L.~{Le Cam} and G.~L. Yang.
\newblock \emph{Asymptotics in Statistics: Some Basic Concepts}.
\newblock Springer, 2000.

\bibitem[Lepskii(1990)]{Lepskii90}
O.~V. Lepskii.
\newblock On a problem of adaptive estimation in gaussian white noise.
\newblock \emph{Theory of Probability and Its Applications}, 35\penalty0
  (3):\penalty0 454--466, 1990.

\bibitem[Lepskii(1991)]{Lepskii91}
O.~V. Lepskii.
\newblock Asymptotically minimax adaptive estimation {I}: upper bounds.
  {O}ptimally adaptive estimates.
\newblock \emph{Theory of Probability and Its Applications}, 36:\penalty0
  682--697, 1991.

\bibitem[Stein(1956{\natexlab{a}})]{Stein56}
C.~Stein.
\newblock Inadmissibility of the usual estimator for the mean of a multivariate
  normal distribution.
\newblock In \emph{Proceedings of the Third Berkeley Symposium on Mathematical
  Statistics and Probability}, pages 197--206, 1956{\natexlab{a}}.

\bibitem[Stein(1956{\natexlab{b}})]{Stein56a}
C.~Stein.
\newblock Efficient nonparametric testing and estimation.
\newblock In \emph{Proceedings of the Third Berkeley Symposium on Mathematical
  Statistics and Probability}, pages 187--195, 1956{\natexlab{b}}.

\bibitem[Tsybakov(1998)]{Tsybakov98}
A.~Tsybakov.
\newblock Pointwise and sup-norm sharp adaptive estimation of functions on the
  {S}obolev classes.
\newblock \emph{Annals of Statistics}, 26\penalty0 (6):\penalty0 2420--2469,
  1998.

\bibitem[van~der Vaart(1997)]{VanDerVaart97}
A.~W. van~der Vaart.
\newblock Superefficiency.
\newblock In D.~Pollard, E.~Torgersen, and G.~Yang, editors, \emph{Festschrift
  for Lucien Le Cam}, chapter~27. Springer, 1997.

\bibitem[van~der Vaart(1998)]{VanDerVaart98}
A.~W. van~der Vaart.
\newblock \emph{Asymptotic Statistics}.
\newblock Cambridge Series in Statistical and Probabilistic Mathematics.
  Cambridge University Press, 1998.

\end{thebibliography}
\bibliographystyle{abbrvnat}

\end{document}